\documentclass[letterpaper,12pt]{article}

\usepackage{microtype,booktabs}
\usepackage{amsmath,amssymb,amsfonts,amsthm,bm}
\usepackage{authblk}
\usepackage{mathrsfs,dsfont}
\usepackage{indentfirst}
\usepackage[pdftex]{color,graphicx}
\usepackage[pdftex,bookmarks,unicode,colorlinks,linkcolor=blue]{hyperref}
\usepackage{tikz}
\usepackage{makecell}
\usepackage{subfigure}
\usetikzlibrary{arrows, decorations.pathmorphing, backgrounds, positioning, fit, petri, automata,cd}
\usepackage{rotating}
\numberwithin{equation}{section}

\title{\Large Transition pathways for a class of degenerate stochastic dynamical systems with L\'evy noise}

\date{}
\author[$\dag$]{\small Ying Chao}
\author[$\ddag$]{{Pingyuan Wei}\thanks{Corresponding author\par \textit{\;\;Email:} yingchao1993@xjtu.edu.cn (Ying Chao),  weipingyuan@pku.edu.cn (Pingyuan Wei)}}
\affil[$\dag$]{\footnotesize School of Mathematics and Statistics, Xi'an Jiaotong University, Xi'an, Shaanxi 710049, China.}
\affil[$\ddag$]{Beijing International Center for Mathematical Research, Peking University, Beijing 100871, China.}

\newtheorem{thm}{Theorem}[section]
\newtheorem{theorem}{Theorem}[section]

\newtheorem{proposition}[thm]{Proposition}
\theoremstyle{remark}

\newtheorem{remark}[thm]{Remark}
\theoremstyle{definition}
\newtheorem{definition}[thm]{Definition}

\begin{document}
\maketitle

\begin{abstract}
This work is devoted to deriving the Onsager--Machlup function for a class of  degenerate stochastic dynamical systems with (non-Gaussian) L\'{e}vy noise  as well as Brownian noise. This is obtained based on the Girsanov transformation and then by a path representation. Moreover, this Onsager--Machlup function may be regarded as  a Lagrangian giving the most probable transition pathways. The Hamilton--Pontryagin principle is essential to handle such a variational problem in degenerate case. Finally, a kinetic Langevin system in which noise is degenerate is specifically investigated analytically and numerically.\\

\noindent\textit{Keywords:} Transition pathways; Onsager--Machlup action functional; degenerate stochastic differential equations; non-Gaussian L\'evy noise

\end{abstract}

%\tableofcontents

\renewcommand{\theequation}{\thesection.\arabic{equation}}
\setcounter{equation}{0}

\section{Introduction}

Environmental noisy fluctuations are inevitable in dynamical systems and may lead  to the transition phenomena between distinct metastable states. Examples of such noise-induced metastable transitions include climate changes \cite{Ditlevsen1999observation}, conformation switching of macromolecules \cite{Weinan2002string}, disease eradication \cite{Schwartz2015noise} and gene transcriptions \cite{Zheng2016transitions}. It is indeed a challenging task to explore the mechanism of transition behaviors in stochastic dynamical systems. Significantly, the Onsager--Machlup (OM) action functional provides an essential tool for assessing the likelihoods of those transitions, and predicting the most probable transition pathway (MPTP) through which the transition occurs \cite{Brocker2019correct}. 

% For many irreversible systems, the loss of detailed balance gives rise to the difficulties of analyzing the asymptotic behavior and transition phenomena.  
OM action functional was first initiated by Onsager and Machlup \cite{Onsager1953} as the probability density functional for a diffusion process with linear drift and constant diffusion coefficient. For stochastic differential equations with (Gaussian) Brownian noise, the OM action functional has been widely investigated during the past few decades, see, e.g., \cite{Stratonovich1971probability, Fujita1982onsager, Ikeda2014stochastic, Shepp1992note, Capitaine1995onsager, Moret2002onsager}. The key point is to express the transition probability of a diffusion process in terms of a functional integral over paths of the process, and the integrand is called OM function.  Then regarding OM function as a Lagrangian \cite{Durr1978}, the most probable transition path of a diffusion process is determined by a variational principle. 

However, certain complex phenomena are not suitable to be modeled as stochastic differential equations with Brownian noise, due to peculiar dynamical features such as heavy-tailed distributions and burst-like events (especially in climate changes \cite{Ditlevsen1999observation}, stock market crashes \cite{Bianchi2010tempered} or tumor metastasis \cite{Hao2014tumor}).  A stochastic process with discontinuous trajectories, e.g., the (non-Gaussian) L\'{e}vy process, appears more appropriate for these phenomena \cite{Applebaum2009, Duan2015introduction}. The related study of OM function for stochastic dynamical systems with L\'evy noise is still under development because of the complicated nonlocal term in the statistical distribution of the noisy fluctuations, but some interesting works are emerging. 
Bardina et al. \cite{Bardina2002asymptotic} dealt with jump functions directly rather than using the Girsanov theorem to obtain the asymptotic evaluation of the Poisson measure for a tube in the path space. Chao and Duan \cite{Chao2019} used the Girsanov transformation to absorb the drift term and thus derived the OM function in one-dimensional nonlinear systems with Brownian noise and L\'{e}vy noise. Subsequently, Hu et al. further generalized this result to high dimensional \cite{Hu2021transition} and infinite dimensional cases \cite{Hu2020}. However, these results do not apply to general degenerate stochastic differential equations.

It should be noted that Aihara and Bagchi \cite{Aihara1999mortensen}, Liu and Gao \cite{Liu2024onsager} derived the OM function for the degenerate stochastic differential equations with Brownian noise successively. This is done by converting stochastic integrals to independent random variables. However, this method is invalid for L\'evy noise case due to the nonequivalence of irrelevance and independence. 

In this present paper, we consider a class of two-dimensional degenerate stochastic dynamical systems under random fluctuations consisting of  Brownian noise and L\'evy noise. On the one hand, we modify the technique in the previous work \cite{Chao2019} to overcome the difficulties cased by L\'evy noise and the degenerate form. More precisely, the OM function is derived by combining the Girsanov transformation, a path representation method with some asymptotic analyses. In this way, we can consider the corresponding variational problem for OM action functional. But different from the non-degenerate case, the most probable pathway connecting distinct states only satisfies an ``implicit" equation (i.e., the Hamilton--Pontryagin equation) instead of the classical Euler--Lagrange one.
On the other hand, we emphasize that our result is consistent with the result of the diffusion processes \cite{Aihara1999mortensen,Liu2024onsager} in the absence of  L\'{e}vy noise. The main difference (in form at least) lies in an extra term of the OM function, depicting the impact of  L\'{e}vy noise.

% The OM function is derived by using the Brownian motion to absorb the drift vector field via the Girsanov transformation for probability measures, and then by a path representation.  Compared with the case of degenerate diffusion processes,  an extra term depicting the impact of  L\'{e}vy noise will appear in  the Onsager-Machlup function. We emphasize that our result is consistent with the result of the diffusion processes \cite{Aihara1999mortensen,Liu2024onsager} in the absence of  L\'{e}vy noise. 

% On the other hand, the Onsager–Machlup function could be considered as  a Lagrangian. By a variational principle, the most probable pathway connecting distinct states satisfies Hamilton-Pontryagin equations (or Euler–Lagrange equations in implicit form). Thus, it reduces to a two-point boundary value problem when we capture the most probable pathway. We will illustrate this for a kinetic Langevin system with quadratic potential under L\'evy noise and Brownian noise. The most probable pathway from one state to another over a finite time interval can be found analytically and numerically.

% An inspiration for this paper goes back to the work \cite{Aihara1999mortensen} by Aihara and Bagchi. We generalize and improve their results to  a class of  degenerate dynamical systems with Brownian noise and L\'{e}vy noise. This will give rise to several difficulties both in analytic and probabilistic aspects as the non-Gaussian pure jump L\'evy noise is present. 

This paper is organized as follows. We first review some notations and basic definitions in Section \ref{sec:2}. Our general theory is in Section \ref{sec:3}. We derive the OM function for a class of degenerate stochastic system with L\'{e}vy noise in Section \ref{sec:31}. Subsequently in Section \ref{sec:32}, we introduce the Lagrangian mechanics for investigating the MPTPs between arbitrary states. In  Section \ref{proof-main}, the proof of our main theorem is presented. Numerical experiments of a kinetic Langevin system (whose noise is degenerate) are in Section \ref{sec:4}, and conclusions follow in Section \ref{sec:5}.

\renewcommand{\theequation}{\thesection.\arabic{equation}}
\setcounter{equation}{0}

\section{Preliminaries}\label{sec:2}
We now recall some necessary notations, introduce a class of stochastic differential equations (SDEs) to be studied, and define the Onsager--Machlup (OM) function as well as OM action functional.

\subsection{Degenerate stochastic differential equations with L\'evy noise}
Let $(\Omega, \mathscr{F}, (\mathscr{F}_t)_{t\geq0}, \mathds{P})$ be a complete filtered probability space, where $\mathscr{F}_t$ is a nondecreasing family of sub-$\sigma$-fields of $\mathscr{F}$ satisfying the usual conditions.
We start with the following degenerate SDE defined on $[0,T]$:
\begin{equation}\label{Equation-2}
\left\{
\begin{array}{l}
dX_t=g(X_t,Y_t)dt, \;\;\; X_0=x_0\in \mathbb{R},\\
dY_t=f(X_t,Y_t)dt+cdW_t+dL_t, \;\;\; Y_0=y_0\in \mathbb{R}, \\
\end{array}
\right.
\end{equation}
%\begin{equation} \label{Equation-1}
%  \begin{split}
%  &dX_t=g(X_t,Y_t)dt, \;\;\; X_0=x_0\in \mathbb{R},   \\
%  &dY_t=f(X_t,Y_t)dt+cdB_t+dL_t, \;\;\; Y_0=y_0\in \mathbb{R},
%  \end{split}
%\end{equation}
where $c$ is a positive constant (referred as the noise intensity), $W_t$ is a standard Brownian motion, and ${L_t}$ is a pure jump L\'{e}vy process with characteristics $(0,0,\nu)$ satisfying $\int_{|\xi|<1}\xi\nu(d\xi)<\infty$. By L\'{e}vy--It\^o decomposition, the L\'evy motion ${L_t}$ can be split into two terms that involve small and large jumps respectively, that is,
\begin{equation}\label{decomposition}
{L_t}= \int_{|x|< 1} x \tilde N(t,dx) + \int_{|x|\ge 1} x N(t,dx),
\end{equation}
where $N(dt,dx)$ is the Poisson random measure on $\mathbb{R}^{+}\times({\mathbb{R}}\backslash \{ 0\})$ and $\tilde N(dt,dx) = N(dt,dx) - \nu (dx)dt$ is the corresponding compensated Poisson random measure with $\nu(A)=\mathds{E}N(1,A)$, $A\in\mathcal{B}({\mathbb{R}}\backslash \{ 0\})$ being the jump measure. 
%In spirits of interlacing \cite[ Page 365] {Applebaum2009}, it makes sense to concentrate on the study of the equation driven by continuous noise interspersed with small jumps. 
% To this end, we shall consider the following degenerate SDE:
% \begin{equation}\label{Equation-2}
% \left\{
% \begin{array}{l}
% dX_t=g(X_t,Y_t)dt, \;\;\; X_0=x_0\in \mathbb{R},\\
% dY_t=f(X_t,Y_t)dt+cdW_t+\int_{|x|<1}x\tilde{N}(dt,dx)+\textcolor{blue}{\int_{|x|\ge 1} x N(dt,dx)}, \;\;\; Y_0=y_0\in \mathbb{R}. \\
% \end{array}
% \right.
% \end{equation}
%$Y_t^L=cW_t+\int_0^t\int_{|x|<1}x\tilde{N}(ds,dx)$.

Denote by $C_b^r(\mathbb{R}^2,\mathbb{R})$ the set of all real-valued $r$th-order continuous differentiable bounded functions defined on $\mathbb{R}^2$ and by $L^2([0,T],\mathbb{R}^2)$ the set of all $\mathbb{R}^2$-valued square integrable functions defined on $[0,T]$. Later in Section \ref{sec:3}, we need to assume that $f\in C_b^2(\mathbb{R}^2,\mathbb{R})$ and $g\in C_b^1(\mathbb{R}^2,\mathbb{R})$. There exists a unique solution to \eqref{Equation-2} up to a maximal stopping time (also called the lifetime) $\tau$ and this solution is adapted and c\`{a}dl\`{a}g (i.e., right-continuous with left limit at each time instant, a.s.), referring to \cite{Applebaum2009,Kunita2019}. For simplicity, we assume that $T\leq \tau$ and we also remark that the solution is global in time (i.e., $\tau=\infty$) if the drift functions $f,g$ satisfies appropriate conditions, e.g., the locally Lipschitz $\&$ one sided linear growth ones \cite{Albeverio2010} or the ones such that the SDE forms a standard Langevin system \cite{Song2020,Song2023}.

Furthermore, we set $Z_t=(X_t,Y_t)^T$, $Z_0=z_0=(x_0,y_0)^T$, and denote by $D_{[0,T]}^{z_0}$ the space of solution paths of (\ref{Equation-2}), that is, % 
$$
D_{[0,T]}^{z_0}=\big\{z:[0,T]\rightarrow \mathbb{R}^2\mid z(t) 
\hbox{ is c\`{a}dl\`{a}g}, \;z(0)=z_0\big\}.
$$
Note that every c\`{a}dl\`{a}g function on $[0,T]$ is bounded. It is clear that $D_{[0,T]}^{z_0}$ is a Banach space, if it is equipped with the following uniform norm $\|\cdot\|$:
$$
\|z\|=\sup_{t\in[0,T]}|z(t)|, \;\; \forall z(t)\in D_{[0,T]}^{z_0}.
$$
%For convenience, by $Y^L$ we mean the stochastic process $Y_t^L=cW_t+L_t$. Similarly, we also denote by $D_{[0,T]}^{0}$ the paths space of $Y_t^L$.

\subsection{Onsager--Machlup function}
%Next, we state the definitions of  Onsager-Machlup function and Onsager-Machlup action functional associated with (\ref{Equation-2}).

In this paper, we concern with the problem of finding the most probable tube of $X_t$. It thus makes sense to ask for the probability that paths lie within the closed tube:
\begin{equation}\label{K-tube}
    K(\phi,\varepsilon)=\big\{
    z\in D_{[0,T]}^{z_0}\mid \phi=(\phi_1,\phi_2)^T\in D_{[0,T]}^{z_0},\; \|z-\phi\|\leq \varepsilon, \;\varepsilon>0
    \big\}.
\end{equation}
Let $\mu_Z$ be the measure induced by the solution process $Z_t$. For given $\varepsilon>0$, we can estimate the probability of this tube by
\begin{equation}\label{probability-tube}
\mu_Z(K(\phi,\varepsilon))=\mathds{P}(\{\omega\in \Omega \mid Z_t(\omega)\in K(\phi,\varepsilon) \}).
\end{equation}
Note that the tube $K(\phi,\varepsilon)$ relies closely upon the choice of the function $\phi(t)$, the key to solve our problem is thus to look for an appropriate function maximizing \eqref{probability-tube}. Based on the viewpoint of Onsager and Machlup, we introduce the following definition:
  
\begin{definition}\label{Definition 2.1} (OM function $\&$ OM action functional) Let $\varepsilon>0$ be given. Consider a tube surrounding a reference path $\phi(t)$. If for $\varepsilon$ sufficiently small we estimate the probability of the solution process $Z_t$ lying in this tube in the form :
$$\mathds{P}(\{\|Z-\phi\|\leq\varepsilon\})\propto C(\varepsilon)\exp\left\{-\int_0^T \text{OM}(\phi, \dot{\phi})dt\right\},$$
then integrand $\text{OM}(\phi, \dot{\phi})$ is called \textit{Onsager--Machlup function}. Where $\propto$ denotes the equivalence relation for $\varepsilon$ small enough. And $I(\phi, \dot{\phi}):=\int_0^T \text{OM}(\phi, \dot{\phi})dt$ is called the  \textit{Onsager--Machlup action functional}.\end{definition}

At this point, once the OM action functional $I(\phi, \dot{\phi})$ is known, we can seek for the function $\phi(t)$ by minimizing this functional (i.e., discussing about its variation), and such a minimizer $\phi(t)$ gives a notation of the most probable transition pathway for the stochastic system \eqref{Equation-2}. 

To achieve the aforementioned goals, it is important for us to distinguish the path space (i.e., the regularities of the functions). In this paper, we introduce the Cameron-Martin $\mathcal{H}$ space as following:
$$
\mathcal{H}=\big\{h:[0,T]\rightarrow \mathbb{R}^2\mid h(t) \hbox{ is absolutely continuous,}\;\dot{h}\in L^2([0,T], \mathbb{R}^2),\;h(0)=0 \big\},
$$
and restrict our arguments on the following assumptions: $\phi(0)=z_0$ and $\phi(t)-\phi(0) \in \mathcal{H}$ for all $t\in [0,T]$.

\renewcommand{\theequation}{\thesection.\arabic{equation}}
\setcounter{equation}{0}

\section{General theory}\label{sec:3}

\subsection{Onsager--Machlup theory for the degenerate
SDE with L\'evy noise}\label{sec:31}

We now state our main theorem associated with OM action functional for (\ref{Equation-2}), and we will present its proof later in Section \ref{proof-main}.

\begin{theorem}\label{theorem 3.1} (Onsager--Machlup theory) Consider a class of degenerate stochastic systems in the form of (\ref{Equation-2}) with the jump measure satisfying $\int_{|\xi|<1}\xi \nu(d\xi)<\infty$ and the initial state $z_0=(x_0,y_0)\in\mathbb{R}^2$. Assume that $f\in C_b^2(\mathbb{R}^2,\mathbb{R})$ and $g\in C_b^1(\mathbb{R}^2,\mathbb{R})$. Then the Onsager--Machlup action functional is given by: % (of course, up to an additive constant)
\begin{align}
    I(\phi, \dot{\phi})=\int_0^T \text{OM}(\phi, \dot{\phi})dt
\end{align}
with the Onsager--Machlup function (of course, up to an additive constant)
\begin{align}
    \text{OM}(\phi, \dot{\phi})=\frac{1}{2}\left|\frac{\dot{\phi_2}(t)-f\left(\phi_1(t),\phi_2(t)\right)+\int_{|\xi|<1}\xi \nu(d\xi)}{c}\right|^2+\frac{1}{2}\frac{\partial f}{\partial y}(\phi_1(t),\phi_2(t)),
\end{align}
where $\phi=(\phi_1,\phi_2)$ is a reference function such that $\phi(0)=z_0$, $\phi(t)-\phi(0)\in \mathcal{H}$ and $\dot{\phi}_1(t)=g(\phi_1(t),\phi_2(t))$. 

Furthermore, the measure of the tube $K(\phi,\varepsilon)$ defined as \eqref{K-tube} can be approximated as follows:
\begin{align}
    \mu_Z(K(\phi,\varepsilon))\propto C(\varepsilon)\exp(-I(\phi, \dot{\phi})),
\end{align}
where the symbol $\propto$ denotes the equivalence relation for $\varepsilon$ small enough.%, and $\mu_{Y^L}$ is the measure induced by the process $Y_t^L=cW_t+L_t$.
%$Y_t^L=cW_t+\int_0^t\int_{|x|<1}x\tilde{N}(ds,dx)$.
\end{theorem}

\begin{remark}\label{Remark 3.1} Compared to degenerate Brownian noise situation, small jumps contribute to the OM action functional and the effect is similar to adding the mean of small jumps to the drift $f(x,y)$. In addition, if the integral $\int_{|\xi|<1}\xi\nu(d\xi)=0$ in the sense of Cauchy principal values, the L\'evy noise will have no effect on the OM action functional. 
\end{remark}

\begin{remark}\label{Remark 3.2} For the case of (degenerate) Brownian noise, Karhunen-Loeve expansion method works well to investigate the OM action functional by converting stochastic integrals to independent random variables. However, this method is invalid for L\'evy noise case due to the nonequivalence of irrelevance and independence. Inspired by our previous work \cite{Chao2019}, we adopt the path representation method here, because it treats the jumps as a whole by  It\^o formula and then these parts could be controlled. Significantly, we could deal with the large jumps in our models because they can be controlled by the bounded variation.

%Theorem \ref{theorem 4.1} presents the OM function for a scalar stochastic differential equation with a general L\'{e}vy process with jump measure $\nu$ satisfying integral $\int_{|\xi|<1}\xi\nu(d\xi)<\infty$. In particular, the result of Theorem \ref{theorem 4.1} is valid for $\alpha$-stable L\'evy motion with $0<\alpha<1$ as $\int_{|\xi|<1}|\xi|\nu_{\alpha, \beta}(d\xi)<\infty$. Thus, $\int_{|\xi|<1}\xi\nu_{\alpha, \beta}(d\xi)<\infty$.
\end{remark}

\begin{remark}\label{Remark 3.3} The conclusion of the Theorem \ref{theorem 3.1} can be generalized to high-dimensional case (with slight modification), as long as $\int_{|\xi|<1}\xi\nu(d\xi)<\infty$ and the primary function for the drift coefficient $f(x,y)$ exists (see, e.g., \cite{Hu2021transition} for sufficient conditions on the existence of the primary function). In fact, for a degenerate SDE on $\mathbb{R}^d\times \mathbb{R}^m$ in the form as follows:
\begin{equation}\label{Equation-1dm}
\left\{
\begin{array}{l}
dX_t=g(X_t,Y_t)dt, \;\;\; X_0=x_0\in \mathbb{R}^d,\\
dY_t=f(X_t,Y_t)dt+cdW_t+dL_t, \;\;\; Y_0=y_0\in \mathbb{R}^m, \\
\end{array}
\right.
\end{equation}
the OM functional is given, up to an additive constant, by
$$I(\phi, \dot{\phi})=\frac{1}{2}\int_0^T\left|\frac{\dot{\phi}_2(t)-f(\phi_1(t),\phi_2(t))+\int_{|\xi|<1}\xi \nu(d\xi)}{c}\right|^2+div_{\phi_2}f(\phi_1(t),\phi_2(t))dt,
$$ 
where $div_{\phi_2}$ denotes the divergence on the second component, and $\phi$ is the function such that $\phi(0)=z_0=(x_0,y_0)\in \mathbb{R}^{d+m}$, $\phi(t)-\phi(0)\in\widetilde{\mathcal{H}}$ and $\dot{\phi}_1(t)=g(\phi_1(t),\phi_2(t))$ with
$$
\widetilde{\mathcal{H}}=\big\{h:[0,T]\rightarrow \mathbb{R}^{d+m}\mid h(t) \hbox{ is absolutely continuous,}\;\dot{h}\in L^2([0,T], \mathbb{R}^{d+m}),\;h(0)=0 \big\}.
$$
\end{remark}

\subsection{The most probable transition pathway between arbitrary states}\label{sec:32}
With the help of Theorem \ref{theorem 3.1}, one can seek for a $\phi(t)$ in certain path space which minimizes the OM functional as the most probable transition pathway (MPTP). If one is interested in transition between arbitrary states, the path space could be
$$
\mathcal{A}=\{\phi\in C^2([0,T]) \mid \phi(0)=z_0, \;\phi(T)=z_T,\; \text{for}\;z_0,\; z_T\in \mathbb{R}^2\}
$$
(in fact, it can be also considered as $\cup_{T>0}\mathcal{A}$). In this way, we need to check that whether there exists (at least) one function $\phi^\ast \in \mathcal{A}$ such that $I(\phi^\ast, \dot{\phi}^\ast)=min_{\phi\in \mathcal{A}}I(\phi, \dot{\phi})$, and $\phi^\ast$ defines the MPTP connecting $z_0$ and $z_T$. 

%A path $z\in C^2([t_0, t_f], \mathbb{R}^2)$ from $z_{t_0}$ to $z_{t_f}$ is said to be minimal if $I[z]\leq I[z+\zeta]$ for every variation $\zeta\in C^2([t_0, t_f], \mathbb{R}^2)$ such that $\zeta(t_0)=\zeta(t_f)=0$.

Further based on the variational principle, the (local) minimizers of an action functional are indeed critical points and satisfy an ``Euler--Lagrange"-like equation equipped with boundary conditions. In analogy to classical mechanics, one can interpret the OM function as a Lagrangian:
$$
\mathcal{L}(\phi_1,\phi_2,\dot{\phi_2})=\frac{1}{2}\left(\frac{\dot{\phi}_2-f(\phi_1,\phi_2)+\int_{|\xi|<1}\xi \nu(d\xi)}{c}\right)^2+\frac{1}{2}\frac{\partial f}{\partial \phi_2}(\phi_1,\phi_2)
$$
with $\dot{\phi_1}=g(\phi_1,\phi_2)$. To  study the MPTP, it reduces to the following constrained variational problem:
\begin{equation}\label{CV}
\left\{
\begin{array}{l}
\delta I=0,\\
I(\phi_1,\phi_2,\dot{\phi_2})=\frac{1}{2}\int_0^T\Big(\frac{\dot{\phi}_2-f(\phi_1,\phi_2)+\int_{|\xi|<1}\xi \nu(d\xi)}{c}\Big)^2+\frac{\partial f}{\partial \phi_2}(\phi_1,\phi_2)dt,\\
\dot{\phi_1}=g(\phi_1,\phi_2),\\
\phi(0)=z_0, \\
\phi(T)=z_T.
\end{array}
\right.
\end{equation}

Note that the Lagrangian $L$ here differs from traditional Lagrangian mechanics, owing to the degenerate noise. To solve the variational problem \eqref{CV}, we introduce an Lagrange multiplier $\lambda(t)$ on the phase space to enforce the constraint. This leads to the following variational principle of Hamilton--Pontryagin \cite{Yoshimura20061,Yoshimura20062}:

\begin{proposition}\label{HP principle}
    (Hamilton--Pontryagin principle) A curve $\gamma=(\phi, \dot{\phi}, \lambda)$ joining $\phi(0)=z_0$ to $\phi(T)=z_T$ satisfies the following ordinary-differential-equation (ODE) system (referred as the Hamilton--Pontryagin equation or general Euler--Lagrange equation in the implicit form):
\begin{equation}\label{Hamilton-Pontryagin}
\left\{
\begin{array}{l}
-\frac{d}{dt}\frac{\partial L}{\partial\dot{\phi_2}}(\phi_1,\phi_2,\dot{\phi_2})+\frac{\partial L}{\partial \phi_2}(\phi_1,\phi_2,\dot{\phi_2})+\lambda\frac{\partial g}{\partial \phi_2}(\phi_1,\phi_2)=0,\\
\frac{\partial L}{\partial \phi_1}(\phi_1,\phi_2,\dot{\phi_2})+\lambda\frac{\partial g}{\partial \phi_1}(\phi_1,\phi_2)+\dot{\lambda}=0,\\
g(\phi_1,\phi_2)-\dot{\phi_1}=0,
\end{array}
\right.
\end{equation}
if $\gamma$ is a critical point of the Hamilton--Pontryagin functional 
$$
J(\phi_1,\dot{\phi_1},\phi_2,\dot{\phi_2},\lambda)=I(\phi_1,\phi_2,\dot{\phi_2})+\int_0^T\langle \lambda,g(\phi_1,\phi_2)-\dot{\phi_1}\rangle dt,
$$
that is, $\delta J=0$.
\end{proposition}

\begin{proof}
    The proof of this proposition is standard, based on calculating the differential of Hamilton--Pontryagin functional, and applying the formula for integration by parts as well as the endpoint conditions. In fact, we only need to calculate that
    \begin{align}
        0=&\frac{d}{d\eta}\Big|_{\eta=0}\int_{0}^{T}J(\phi_1+\eta\zeta_1,\dot{\phi_1}+\eta\dot{\zeta_1}, {\phi_2}+\eta{\zeta_2},\dot{\phi_2}+\eta\dot{\zeta_2},\lambda+\eta\zeta_3)dt\notag\\
        =&\int_0^T \bigg[
        \frac{\partial \mathcal{L}}{\partial \phi_1}(\phi_1,\phi_2,\dot{\phi_2})\zeta_1+\frac{\partial \mathcal{L}}{\partial \phi_2}(\phi_1,\phi_2,\dot{\phi_2})\zeta_2+\frac{\partial \mathcal{L}}{\partial \dot{\phi_2}}(\phi_1,\phi_2,\dot{\phi_2})\dot{\zeta_2}\notag\\
        &+\lambda \frac{\partial g}{\partial \phi_1}(\phi_1,\phi_2)\zeta_1-\lambda\dot{\zeta_1}+\lambda \frac{\partial g}{\partial \phi_2}(\phi_1,\phi_2)\zeta_2
        +(g(\phi_1,\phi_2)-\dot{\phi_1})\zeta_3
        \bigg]dt\notag\\
        =&\int_0^T\bigg\{\bigg[
        \frac{\partial L}{\partial \phi_1}(\phi_1,\phi_2,\dot{\phi_2})+\lambda\frac{\partial g}{\partial \phi_1}(\phi_1,\phi_2)+\dot{\lambda}
        \bigg]\zeta_1\notag\\
        &+\bigg[
        -\frac{d}{dt}\frac{\partial L}{\partial\dot{\phi_2}}(\phi_1,\phi_2,\dot{\phi_2})+\frac{\partial L}{\partial \phi_2}(\phi_1,\phi_2,\dot{\phi_2})+\lambda\frac{\partial g}{\partial \phi_2}(\phi_1,\phi_2)
        \bigg]\zeta_2\notag\\
        &+(g(\phi_1,\phi_2)-\dot{\phi_1})\zeta_3 \bigg\}dt,\notag
    \end{align}
    for all variations  $\zeta_1,\; \zeta_2,\; \zeta_3 \in C^2([0, T], \mathbb{R})$ with $\zeta_i(0)=\zeta_i(T)=0$, $i=1,2,3$.
\end{proof}

Eventually, we conclude that each MPTP solves a boundary value problem of the Hamilton--Pontryagin equation \eqref{Hamilton-Pontryagin}. It should be noted that such a problem does not always have a solution. Interestingly, for some special cases (e.g., kinetic Langevin systems with quadratic potentials; see Section \ref{sec:4}), analytical solution to this ODE system \eqref{Hamilton-Pontryagin} may exist, and the constrained ``minimization" could be solved analytically. Moreover, we emphasize that the MPTP we have found is not the real path for the original stochastic system (\ref{Equation-2}), and it captures practical paths of the largest probability around its neighborhood in the sense of Theorem \ref{theorem 3.1}.

\begin{remark}\label{rmk3-5}
    For some practical models, one may be more interested in the MPTP connecting configurations (instead of states on the phase space). The approach presented here is still applicable with the following modification: Firstly, fix $y(0)$ and solve the variational problem without end point constraint by Proposition \ref{HP principle}; Secondly, optimize among the solutions satisfying the end point constraint; Thirdly, optimize with respect to $y(0)$. In particularly, if $\frac{\partial f}{\partial \phi_2}$ is a constant and there is no additional requirement (we remark that in molecular dynamics, for example, one may take the Gibbs-Boltzmann distribution of kinetic energy into account and consider an alternative version of the action), the solution would usually be the Newtonian path, for which the initial velocity is big enough to overcome all energy barriers so that the square in the integral of $I$ is just equal to 0.
\end{remark}

\subsection{Proof of the Theorem \ref{theorem 3.1}}\label{proof-main}

Afer presenting our general theory in Section \ref{sec:31} $\&$ \ref{sec:32}, we now turn to proving the main theorem by the following three steps. %For convenience, in the proof below $C$ stands for an unspecified constant.

%\begin{proof}[ ]
\noindent {\bf Step 1: Applying Girsanov transformation to absorb drift term}\par
\smallskip
Consider the following auxiliary SDE for $Z_t^*=(X_t^*,Y_t^*)^T$ with respect to the probability space $(\Omega, \mathscr{F}, (\mathscr{F}_t)_{t\geq0}, \mathds{P})$:
\begin{equation}\label{Auxiliary-Eq}
\left\{
\begin{array}{l}
dX_t^*=g(X_t^*,Y_t^*)dt, \;\;\; X_0^*=x_0\in \mathbb{R},\\
dY_t^*=\dot{\phi_2}(t)dt+cdW_t+dL_t, \;\;\; Y_0^*=y_0\in \mathbb{R}. \\
\end{array}
\right.
\end{equation}
Denote by $Y^L$ the stochastic process $Y_t^L=cW_t+L_t$ for convenience. Clearly, we can solve the second equation in \eqref{Auxiliary-Eq} and obtain that $Y_t^*=\phi_2(t)+Y_t^L$. On the other hand, we put
\begin{equation}\label{Mt}
    M_t=\exp\left[\int_0^t F(X_s^*,Y_s^*,\dot{\phi}_2(s))dW_s-\frac{1}{2}\int_0^t \left|F(X_s^*,Y_s^*,\dot{\phi}_2(s))\right|^2 ds\right]
\end{equation}
with
$$
F(x,y,\dot{\phi}_2(t))=\frac{1}{c}\big(f(x,y)-\dot{\phi_2}(t)\big)%=\frac{1}{c}\big(f(X_t^*,\phi_2(t)+Y_t^L)-\dot{\phi_2}(t)\big)
$$
and all $ t\leq T$. By applying the Theorem \cite[Theorem 1.4]{Ishikawa2023stochastic} with respect to Girsanov transformation, the process $M_t$ is martingale, and
$$
\mathds{Q}(A)=\int_A M_T(\omega)d\mathds{P}(\omega)
$$
defines a probability measure on $\Omega$ such that $\mathds{Q}(A)=\int_A M_t(\omega)d\mathds{P}(\omega)$ for all $A\in \mathscr{F}_t$. Furthermore, under the new filtered probability space $(\Omega, \mathscr{F}, (\mathscr{F}_t)_{t\geq0}, \mathds{Q})$, the process $W_t^*:=W_t-\int_0^t F(X_s^*,Y_s^*,\dot{\phi}_2(s))ds$ is indeed a Brownian motion and the random measure $\tilde{N}(dt,dx)$ is still a compensated Poisson one with the jump measure $\nu$. Hence, the stochastic differential representation for $Z_t^*$ with respect to $(\Omega, \mathscr{F}, (\mathscr{F}_t)_{t\geq0}, \mathds{Q})$ is given by
\begin{equation}\label{Auxiliary-Eq-Q}
\left\{
\begin{array}{l}
dX_t^*=g(X_t^*,Y_t^*)dt, \;\;\; X_0^*=x_0\in \mathbb{R},\\
dY_t^*=f(X_t^*,Y_t^*)dt+cdW_t^*+dL_t. \\
\end{array}
\right.
\end{equation}
Compare the original SDE \eqref{Equation-2} for $(Z_t,\mathds{P})$ with the above SDE \eqref{Auxiliary-Eq-Q} for $(Z_t^*,\mathds{Q})$, and denote by $\mu_Z$ and $\mu_{Z^*}^{\mathds{Q}}$ the measures induced by these two SDEs respectively. We can find that $\mu_Z=\mu_{Z^*}^{\mathds{Q}}$, according to the uniqueness in distribution \cite [Page 410] {Applebaum2009}. 

As a result, we have the following expression for Radon--Nikodym derivative:
\begin{equation}\label{Radon-Nikodym}
   \frac{d\mu_Z}{d\mu_{Z^*}}[Z_t^*(\omega)]=\frac{d\mu_{Z^\ast}^{\mathds{Q}}}{d\mu_{Z^*}}[Z_t^*(\omega)]=\frac{d\mathds{Q}}{d\mathds{P}}(\omega)=M_T(\omega). 
\end{equation}
In this way, for any $\varepsilon>0$,
\begin{equation}\label{EP}
\mathds{P}(\|Z-\phi\|\leq\varepsilon)=\mathds{Q}(\|Z^*-\phi\|\leq\varepsilon)=\int_{\|Z^*-\phi\|\leq\varepsilon}\frac{d\mathds{Q}}{d\mathds{P}}d\mathds{P}(\omega)=\mathds{E}\big[M_T \mathds{1}_{\{\|Z^*-\phi\|\leq\varepsilon\}}\big].
\end{equation}

To estimate (\ref{EP}), we give a result about the controllability on small ball probability, i.e., there exists a positive constant $K$ such that for a.s. $\omega\in\Omega$,
\begin{equation}\label{controllability}
\|X^*-\phi_1\|\leq K \|Y^*-\phi_2\|=K\|Y^L\|.
\end{equation}
In facts, it follows from (\ref{Auxiliary-Eq}) and $\phi_1(t)=x_0+\int_0^tg(\phi_1(s),\phi_2(s))ds$ that
$$
|X_t^*-\phi_1(t)|^2\leq K_1\int_0^T\{|X_s^*-\phi_1(s)|^2+|Y_s^*-\phi_2(s)|^2\}ds.
$$
Based on Gronwall's inequality and the fact that $Y_t^*=\phi_2(t)+Y_t^L$, we have
$$|X_t^*-\phi_1(t)|^2\leq e^{K_1t}\int_0^t|Y_s^*-\phi_2(s)|^2ds\leq K_2 \sup_{t\in[0,T]}|Y_t^*-\phi_2(t)|^2=K_2 \sup_{t\in[0,T]}|Y_t^L|^2.$$
So the desired result (\ref{controllability}) holds. In this way, we infer that 
$$\|Z^*-\phi\|=O(\varepsilon) \iff \|Y^L\|=O(\varepsilon),$$
and thus rewrite (\ref{EP}) as
\begin{equation}\label{EPR}
\mathds{P}(\|Z-\phi\|\leq\varepsilon)=\mathds{E}\big[M_T \mathds{1}_{\{\|Y^L\|\leq\varepsilon\}}\big].
\end{equation}
We highlight that the probability $\mathds{P}(\|Y^L\|\leq \varepsilon)$ is not related to the drift functions, and the drift information is only contained  in $M_T$.% (in other words, it has been absorbed into $M_T$).

% \smallskip
% \noindent {\bf Step 2: Dominating the Possion integral by bounded variation}\par
\smallskip
\noindent {\bf Step 2: Representing the Radon--Nikodym derivative in terms of path integrals}\par
\smallskip

Keep in mind that now our aim is to study the limiting behaviors of (\ref{EPR}) as $\varepsilon$ tends to $0$. Before that, we need to deal with
$M_T$ given by the Radon--Nikodym derivative \eqref{Radon-Nikodym}, based on the path representation method. A key point is to transform the stochastic integral in \eqref{Mt} using the It\^o's formula \cite[Theorem 4.4.7]{Applebaum2009}.

We now set the potential function
$$
V(x,y,\dot{\phi}_2(t))=\frac{1}{c}\int^yF(x,u,\dot{\phi}_2(t))du
$$
and calculate its differential at $(X_t^*,Y_t^*, \dot{\phi_2}(t))$ by It\^o's formula:
\begin{align}\label{Ito-V}
    &dV(X_t^*,Y_t^*, \dot{\phi_2}(t))\notag\\
    =&\left(\frac{\partial V}{\partial t}+g\cdot\frac{\partial V}{\partial x}+\dot{\phi_2}(t)\cdot \frac{1}{c}F+\frac{1}{2}\frac{\partial f}{\partial y}-\int_{|\xi|<1}\xi \nu(d\xi)
    \cdot\frac{1}{c}F\right)\bigg|_{(X_t^*,Y_t^*, \dot{\phi_2}(t))}dt\notag\\
    &+ F(X_t^*,Y_t^*, \dot{\phi_2}(t))dW_t\notag\\
    &+\int_{\mathbb{R}\backslash \{0\}}\left[\int_{Y^*(t)}^{Y^*(t)+\xi}\frac{1}{c}F(X^*(t),u,\dot{\phi_2}(t))du\right] N(dt,d\xi),
    %&+\sum_{0\leq t \leq T}\left[V(X_t^\ast,Y_t^\ast,\dot{\phi_2}(t))-V(X_{t-}^\ast,Y_{t-}^\ast,\dot{\phi_2}(t-))\right]\notag\\ 
    %&+\sum_{0\leq t \leq T}\left[\int_{Y^*(t-)}^{Y^*(t)}\frac{1}{c}F(X^*(t),u,\dot{\phi_2}(t))du\right]\notag\\ 
    %:=& B(X_t^*,Y_t^*, \dot{\phi_2}(t))dt+F(X_t^*,Y_t^*, \dot{\phi_2}(t))dW_t+\mathcal{S},
\end{align}
where $\int_{|\xi|<1}\xi\nu(d\xi)<\infty$ under our assumptions. For convenience, we denote by $B$ the function before $dt$ in \eqref{Ito-V}. Integrating from $0$ to $T$ on both sides and applying Proposition 4.4.8 (or Theorem 4.4.10) in \cite{Applebaum2009}, the above equation \eqref{Ito-V} is equal to
\begin{align}\label{int-F}
    \int_0^TF(X_s^*,Y_s^*, \dot{\phi_2}(s))dW_s
    =&V(X_T^*,Y_T^*, \dot{\phi_2}(T))-V(x_0,y_0, \dot{\phi_2}(0))-\int_0^TB(X_s^*,Y_s^*, \dot{\phi_2}(s))ds\notag\\
    &-\sum_{0\leq t \leq T}\bigg[\int_{Y^*(t-)}^{Y^*(t)}\frac{1}{c}F(X^*(t),u,\dot{\phi_2}(t))du\bigg].
\end{align}
Therefore, by substituting the relation \eqref{int-F} into \eqref{Mt} (with $t=T$) and denoting by $z^\ast(t)$ and $y^L(t)$ the path functions for $Z^\ast$ and $Y^L$ respectively, we infer that the Radon--Nikodym derivative satisfies the following functional property on $D_{[0,T]}^{z_0}$:
\begin{align}\label{RND-2}
    \mathcal{M}[z^\ast(t)]=&\frac{d\mu_Z}{d\mu_{Z^*}}[z^\ast(t)]\notag\\
    =&\exp\bigg\{
    V(z^\ast(T),\dot{\phi_2}(T))-V(z_0,\dot{\phi_2}(0))-\int_0^T(\frac{1}{2}F^2+B)(z^\ast(s), \dot{\phi_2}(s))ds\notag\\
    &-\sum_{0\leq t \leq T}\bigg[\int_{y^*(t-)}^{y^*(t)}\frac{1}{c}F(x^*(t),u,\dot{\phi_2}(t))du\bigg]
    \bigg\},
\end{align}
for %\textcolor{magenta}{
any $z^\ast(t)=(x^\ast(t),y^\ast(t))\in D^{z_0}_{[0,T]}$. Note that $y^\ast(t)=\phi_2(t)+y^L(t)$, and for the $K(0,\varepsilon)=\{z\in D_{[0,T]}^{0}\mid \|z\|\leq\epsilon,\epsilon>0\}$, the induced measure $\mu_{Y^L}$ of $Y^L$ satisfying $\mu_{Y^L}(K(0,\varepsilon))=\mathds{P}(\|Y^L\|\leq \varepsilon)$. We further rewrite (\ref{EPR}) as follows
\begin{equation}\label{EPRR}
\mathds{P}(\|Z-\phi\|\leq\varepsilon)=\int_{\|Y^L\|\leq\varepsilon}M_T(\omega) d\mathds{P}(\omega)=\int_{K(0,\varepsilon)}\mathcal{M}[z^\ast(t)]d\mu_{Y^L}(y^L),
\end{equation}
in which the integrand $\mathcal{M}[z^\ast(t)]$ given in \eqref{RND-2} does not depend on stochastic integrals with respect to the Brownian motion $W_t$.

\smallskip
\noindent {\bf Step 3: Estimating the expression by Taylor expansion and bounded variation.}\par
\smallskip
We point out that the expression \eqref{RND-2} contains a Riemann integral with respect to the time variable $t$ and an infinite series about jumps of the path function $y^\ast(t)$. In this step, we would like to estimate the former one by Taylor expansion, and discuss about the latter one with the help of the bounded variation of L\'evy motion $L_t$.

On the one hand, we expand the exponent of \eqref{RND-2} into a Taylor series around $y^L(t)=0$ and split iff the terms of zero order. The remaining terms can be made arbitrarily small if we choose $\varepsilon$ small enough, since $\|y^L(t)\|\leq \varepsilon$ holds for $y^L(t)\in K(0,\varepsilon)$. As discussed in \eqref{controllability}, we consider $z^\ast(t)=(x^\ast(t),y^\ast(t))\in D_{[0,T]}^{z_0}$ with
$$
\|x^*(t)-\phi_1(t)\|\leq \varepsilon, \; \|y^*(t)-\phi_2(t)\|=\|y^L(t)\|\leq \varepsilon.
$$
Provided that the drift function $f(x,y)$ is $C_b^2$ and the potential function $V$ is at least $C^2$ (in $x$, $y$, uniformly in $t\in[0,T]$), we obtain an expansion of $V$ (in the first-order of $\varepsilon$):
\begin{align}
    V(z^\ast(t),\dot{\phi_2}(t))=&V(\phi(t), \dot{\phi_2}(t))+(x^*(t)-\phi_1(t))\frac{\partial V}{\partial x}(\phi(t), \dot{\phi_2}(t))
    \notag\\
    &+(y^*(t)-\phi_2(t))\frac{\partial V}{\partial y}(\phi(t), \dot{\phi_2}(T))+o(\varepsilon)\notag\\
    =&V(\phi(t), \dot{\phi_2}(t))+O(\varepsilon),
\end{align}
for each $t\in[0,T]$, and thus have
\begin{align}\label{estimate1}
    V&(z^\ast(T), \dot{\phi_2}(T))-V(z^\ast(0),\dot{\phi_2}(0))
    =\int_0^T\frac{\partial }{\partial t}V(\phi(t),\dot{\phi}_2(t))dt+O(\varepsilon)\notag\\
    %=&\int_0^T\left(\frac{\partial V}{\partial t}+\frac{\partial V}{\partial x}\dot{\phi_1}(t)+\frac{\partial V}{\partial y}\dot{\phi_2}(t)\right)(\phi(t),\dot{\phi}_2(t))dt+o(\varepsilon)\notag\\
    =&\int_0^T\left(\frac{\partial V}{\partial t}+\frac{\partial V}{\partial x} \cdot g+\frac{1}{c}F\cdot\dot{\phi_2}(t)\right)(\phi(t),\dot{\phi}_2(t))dt+O(\varepsilon).%\notag\\
\end{align}
Similarly, we can obtain the expansions of $F^2$ and $B$ and then calculate that
\begin{align}\label{estimate2}
    &\int_0^T(\frac{1}{2}F^2+B)(z^\ast(s), \dot{\phi_2}(s))ds\notag\\
    =&\int_0^T\left(
    \frac{1}{2}F^2+\frac{\partial V}{\partial t}+g\cdot \frac{\partial V}{\partial x}+\dot{\phi_2}(t)\cdot \frac{1}{c}F+\frac{1}{2}\frac{\partial f}{\partial y}-\int_{|\xi|<1}\xi \nu(d\xi)
    \cdot\frac{1}{c}F
    \right)(\phi(s), \dot{\phi_2}(s))ds\notag\\
    &+O(\varepsilon).%\notag\\
    %&=-\int_0^T\left[\left(\frac{1}{2}+\frac{1}{c}\dot{\phi_2}(t)-\frac{1}{c}\int_{|\xi|<1}\xi \nu(d\xi)\right)F+\frac{\partial V}{\partial t}+g\cdot \frac{\partial V}{\partial x}+\frac{1}{2}\frac{\partial f}{\partial y}\right](\phi(s), \dot{\phi_2}(s))ds.
\end{align}

On the other hand, we denote by $\Delta y^*(t)=y^*(t)-y^*(t-)$ the jump of $y^*$ at the point $t$. Note that $Y^*(t)=\phi_2(t)+Y^L(t)=\phi_2(t)+cW(t)+L(t)$ with $\phi_2(t)+cW(t)$ being continuous in $t\in[0,T]$, we claim that $Y_t^\ast$, $Y_t^L$ and $L_t$ have the same bounded variation. In fact, according to Theorem 2.3.14 (or Example 2.3.15) and Theorem 2.4.25, for a.s. $\omega\in\Omega$, the finite variation $V_{[0,T]}(\omega)$ of $L_t$ exists, and 
$$
\sum_{0\leq t \leq T}|\Delta Y^\ast(t)|=\sum_{0\leq t \leq T}|\Delta Y^L(t)|=\sum_{0\leq t \leq T}|\Delta L(t)|\leq V_{[0,T]}(\omega)<\infty.
$$
Under our assumptions, the function $F$ is $C_b^2$ so that $\|F\|$ is finite and
\begin{align}\label{estimate3}
    \bigg|\sum_{0\leq t \leq T}\bigg[\int_{y^*(t-)}^{y^*(t)}\frac{1}{c}F(x^*(t),u,\dot{\phi_2}(t))du\bigg]\bigg |\leq \frac{1}{c}\|F\|\sum_{0\leq t \leq T} |\Delta y^*(t)|<\infty,
\end{align}
for any $z^\ast(t)\in D^{z_0}_{[0,T]}$.

Consequently, by combining \eqref{RND-2}, \eqref{EPRR} with the estimations  \eqref{estimate1}, \eqref{estimate2} and \eqref{estimate3}, we conclude that 
\begin{align}
    \mathds{P}(\{\|Z-\phi\|)=&\int_{K(0,\varepsilon)}\exp\left(-\hat{I}(\phi, \dot{\phi}_2)+C(y^L)+O(\varepsilon)\right) d\mu_{Y^L}(y^L)\notag\\
    =&C_\varepsilon \mu_{Y^L}(K(0,\varepsilon)) \exp(-\hat{I}(\phi, \dot{\phi}_2))
\end{align}
with $C_\varepsilon$ being a constant with respect to $\varepsilon$, and 
$$
\hat{I}(\phi, \dot{\phi_2})=\frac{1}{2}\int_0^T\left[\left(
    F-\frac{1}{c}\int_{|\xi|<1}\xi \nu(d\xi)
    \right)^2+\frac{\partial f}{\partial y}\right](\phi(s), \dot{\phi_2}(s))ds.
$$
Clearly, $I(\phi, \dot{\phi})=\hat{I}(\phi, \dot{\phi_2})$ is the desired OM action functional by Definition \ref{Definition 2.1}. The proof of Theorem \ref{theorem 3.1} is thus complete. \hfill $\qedsymbol$

\renewcommand{\theequation}{\thesection.\arabic{equation}}
\setcounter{equation}{0}

%\section{Langevin systems: MPTPs based on the fourth-order Hamilton-Pontryagin variation  }\label{sec:4}
\section{Applications to Langevin systems}\label{sec:4}

This section considers a specific case of stochastic mechanical systems incorporating the degenerate noise, modeled as a second-order and underdamped kinetic Langevin system:
\begin{equation}\label{Langevin1}
    \ddot{X}+\gamma\dot{X}=-U'(X)+\sqrt{\mu} \gamma^{\frac{1}{2}}\dot{W}(t)+\dot{L}(t),
\end{equation}
where $U\in C^3( \mathbb{R},\mathbb{R})$ is the potential function, $\gamma\in\mathbb{R}^+$ is a damping coefficient, $\mu=2k_BT$ stands for the thermal temperature, $W(t)$ is a  standard Brownian motion on $\mathbb{R}$, 
and $L(t)$ is an asymmetric $\alpha$-stable L\'evy process with the generating triplet $(0,0,\nu_{\alpha, \beta})$.
%and $\tilde N$ is the compensated Poisson process corresponding to an asymmetric $\alpha$-stable L\'evy process $L(t)$ with the generating triplet $(0,0,\nu_{\alpha, \beta})$. 
To ensure that Theorem \ref{theorem 3.1} is adaptive, we assume that $\alpha\in(0,1)$ and $\beta\in[-1,1]$. In fact, the jump measure $\nu_{\alpha, \beta}$ is given by  $$\nu_{\alpha, \beta} (d\xi)=c_1|\xi|^{-1-\alpha}\mathds{1}_{\{0<\xi<\infty\}}d\xi+c_2|\xi|^{-1-\alpha}\mathds{1}_{\{-\infty<\xi<0\}}d\xi$$
%c_1\geq0$, $c_2\geq0$, $c_1+c_2>0$
with $\beta=\frac{c_1-c_2}{c_1+c_2}$, $c_1=k_{\alpha}\frac{1+\beta}{2}$, and $c_2=k_{\alpha}\frac{1-\beta}{2}$, where
$$
k_{\alpha}=
\left \{
  \begin{array}{ll}
    \frac{\alpha(1-\alpha)}{\Gamma(2-\alpha)\cos(\frac{\pi\alpha}{2})}, &  \hbox{if $\alpha<1$,} \\
    \frac{2}{\pi}, & \hbox{if $\alpha=1$.}
  \end{array}
\right.
$$
It is clear that $\int_{|\xi|<1}|\xi|\nu_{\alpha, \beta}(d\xi)<\infty$ if and only if $\alpha<1$. Moreover, we have $\int_{|\xi|<1}\xi\nu_{\alpha,\beta}(d\xi)=\frac{\alpha \beta}{\Gamma(2-\alpha)\cos(\frac{\pi\alpha}{2})}:=\Lambda_{\alpha,\beta}$ for $\alpha<1$.

\subsection{(Parametric) Hamilton--Pontryagin equations v.s. (4th-order) Euler--Lagrange equations}

We now introduce $Y=\dot{X}$ and rewrite \eqref{Langevin1} into the form of \eqref{Equation-2} with $g(x,y)=y$, $f(x,y)=-\gamma y-U'(x)$ and $c=\sqrt{\mu}\gamma^{\frac{1}{2}}$. We remark that $X$, $Y$ and $Z=(X,Y)$ are referred as position (or configuration), velocity (or momentum) and phase variables, respectively. According to Theorem \ref{theorem 3.1}, the OM function of this system is
$$
OM(\phi,\dot{\phi})=\frac{1}{2}\left(\frac{\dot{\phi}_2+\gamma\phi_2+U'(\phi_1)+\int_{|\xi|<1}\xi \nu_{\alpha,\beta}(d\xi)}{c}\right)^2-\frac{\gamma}{2}, \;\;\text{with}\;\; \dot{\phi_1}=\phi_2,
$$
and the OM action functional is $I(\phi, \dot{\phi})=\int_0^T \text{OM}(\phi, \dot{\phi})dt$. This leads to the following Hamilton--Pontryagin (HP) equation containing a parametric function $\lambda(t)$ by Proposition \ref{HP principle} (i.e., Hamilton--Pontryagin principle):
\begin{equation}\label{HP-Lagevin}
\left\{
\begin{array}{l}
-\frac{1}{c^2}(\ddot{\phi_2}+U''(\phi_1)\dot{\phi_1}+\gamma\dot{\phi_2})+\frac{\dot{\phi_2}+\gamma\phi_2+U'(\phi_1)+\int_{|\xi|<1}\xi \nu(d\xi)}{c^2}\gamma+\lambda=0,\\
\frac{\dot{\phi_2}+\gamma\phi_2+U'(\phi_1)+\int_{|\xi|<1}\xi \nu(d\xi)}{c^2}U''(\phi_1)+\dot{\lambda}=0,\\
\dot{\phi_1}=\phi_2,\\
\phi(0)=z_0=(x_0,y_0)^T,\\
\phi(T)=z_T=(x_T,y_T)^T.
\end{array}
\right.
\end{equation}

For this special system, we may also replace $\phi_2$ by $\dot{\phi_1}$ and focus our arguments on the position variable $\phi_1$ as well as its time derivatives. Associated with the OM function $OM(\phi,\dot{\phi})$, we introduce a Lagrangian with respect to $\phi_1$, $\dot{\phi_1}$ and $\ddot{\phi_1}$:
\begin{equation*}%\label{L-special}
\mathcal{L}(\phi_1,\dot{\phi_1},\ddot{\phi_1})=\frac{\left|\ddot{\phi_1}+\gamma \dot{\phi_1}+U'(\phi_1)+\int_{|\xi|<1}\xi \nu_{\alpha,\beta}(d\xi)\right|^2}{2\mu\gamma}-\frac{\gamma}{2}.
\end{equation*}
Accordingly, we define the action functional as $I=\int_0^T\mathcal{L}(\phi_1,\dot{\phi_1},\ddot{\phi_1})dt$. By the high-order variational principle \cite{Riahi1972lagrangians}, a minimizer $\phi_1(t)$ of $I$ solves the Euler--Lagrange (EL) equation
\begin{equation*} %\label{EL}
\frac{\delta I}{\delta \phi_1}=\frac{\partial\mathcal{L}}{\partial \phi_1}-\frac{d}{dt}\frac{\partial\mathcal{L}}{\partial\dot{\phi_1}}+\frac{d^2}{dt^2}\frac{\partial\mathcal{L}}{\partial\ddot{\phi_1}}=0,
\end{equation*}
equipped with the boundary conditions. More precisely, we are indeed dealing with a fourth-order nonlinear boundary value problem
with respect to $\phi_1$:
\begin{equation}\label{EL-concrete}
\left\{
\begin{array}{l}
\ddddot{\phi_1}+\ddot{\phi_1}(2U''(\phi_1)-\gamma^2)+\dot{\phi}_1^2U'''(\phi_1)+\left(U'(\phi_1)+\int_{|\xi|<1}\xi \nu_{\alpha,\beta}(d\xi)\right)U''(\phi_1)=0,\\
\phi_1(0)=x_0,\;\dot{\phi_1}(0)=y_0,\\
\phi_1(T)=x_T,\;\dot{\phi_1}(T)=y_T,
\end{array}
\right.
\end{equation}
which is consistent with \eqref{HP-Lagevin}; See the diagram below.
\begin{center}
\begin{tikzpicture}
  \path (0,0) node (a) {$OM(\phi,\dot{\phi})$}
        (7,0) node (b) {$L(\phi_1,\dot{\phi_1},\ddot{\phi_1})$}
        (0,-3) node[align=center] (c) {(parametric) \\ HP equation}
        (7,-3) node[align=center] (d) {(4th-order) \\EL equation.};
  \draw[->]   (a) -- node[font=\footnotesize,above] { $\phi_2=\dot{\phi_1}$}(b);
  \draw[->]   (a) -- node[font=\footnotesize,right,align=center,midway] { Hamilton--\\Pontryagin\\principle}(c);
  \draw[->]   (b) -- node[font=\footnotesize,right,align=center,midway] { high-order\\ variational\\principle}(d);
  \draw[->]   (c) -- node[font=\footnotesize,above,align=center,midway] {canceling $\phi_2\&\lambda$}(d);
\end{tikzpicture}  
\end{center}

\subsection{The MPTP based on an analytical solver}

It should be pointed out that, in both \eqref{HP-Lagevin} and \eqref{EL-concrete}, the initial and final velocities are given. But in practice, the information of velocities maybe unknown, and we thus need to seek for a global MPTP between configurations, that is, the one from $\phi_1(0)=x_0$ to $\phi_1(T)=x_T$ here. As mentioned in Remark \ref{rmk3-5}, we can fix $\phi_2(0)=\dot{\phi_1}(0)$ and consider one set of sufficient initial conditions: $\phi_1(0),\dot{\phi_1}(0),\ddot{\phi_1}(0),\dddot{\phi_1}(0)$. Clearly, the first two are known, the last two should satisfy $\phi_1(T)=x_T$, and the solution will be a function of them. Hence, we optimize the action under this constraint to get the MLTP given $\dot{\phi_1}(0)$, and then further minimize the action as a function of $\dot{\phi_1}(0)$.

In fact, for this special system, things are even simpler. Notice that the action functional reaches $I_{\text{min}}=-\frac{\gamma T}{2}$ if and only if
\begin{align}\label{second-order}
\ddot{\phi_1}+\gamma\dot{\phi_1}+U'(\phi_1)+\int_{|\xi|<1}\xi \nu_{\alpha,\beta}(d\xi)=0.
\end{align}
We thus only need to solve \eqref{second-order} with boundary conditions: $\phi_1(0)=x_0$ and $\phi_1(T)=x_T$. In addition, when the potential is quadratic, analytical solutions to the ODE systems (\ref{EL-concrete}) and (\ref{second-order}) exist, and the MPTPs could be solved analytically. 

From now on, we set $U(x)=-\frac{1}{2}x^2$. According to the variation of constants, the explicit solution for problem \eqref{EL-concrete} is 
\begin{equation}\label{example-solution1}
\phi_1(t)=C_1e^{\lambda_1t}+C_2e^{\lambda_2t}+C_3e^{\lambda_3t}+C_4e^{\lambda_4t}+\Lambda_{\alpha,\beta},
\end{equation}
where $\Lambda_{\alpha,\beta}=\frac{\alpha \beta}{\Gamma(2-\alpha)\cos(\frac{\pi\alpha}{2})}$, $\lambda_i=\mp\sqrt{\frac{2+\gamma^2\mp\gamma\sqrt{\gamma^2+4}}{2}}$ for $i=1,2,3,4$ and $\lambda_1<\lambda_2<\lambda_3<\lambda_4$, and 
$C_1=\frac{det(A_0,A_2,A_3,A_4)}{det(A_1,A_2,A_3,A_4)}$, $\cdots$, $C_4=\frac{det(A_1,A_2,A_3,A_0)}{det(A_1,A_2,A_3,A_4)}$ with $A_0=(x_0-\Lambda_{\alpha,\beta},x_1-\Lambda_{\alpha,\beta},y_0,y_T)^T$, $A_i=(1,e^{\lambda_iT},\lambda_i,\lambda_ie^{\lambda_iT})^T$. Clearly, the MPTP does depend on the choices of $\dot{\phi_1}(0)=y_0$ and $\dot{\phi_1}(T)=y_T$. That is, given different $y_0$ and $y_T$, we would obtain different MPTPs as shown in Figure 1. To get a global MPTP between the configurations, we should substitute \eqref{example-solution1} back to the action $I$ and minimize $I$ with respect to $y_0$ and $y_T$. Here, the explicit expression of such a global MPTP is
\begin{equation}\label{example-solution2}
\phi_1^\star(t)=C_1^\star e^{\lambda_1^\star t}+C_2^\star e^{\lambda_2^\star t}+\Lambda_{\alpha,\beta},
\end{equation}
which can be obtained by solving the boundary problem of \eqref{second-order} directly, and in which $\lambda_{1,2}^\star=\frac{-\gamma\mp\sqrt{\gamma^2+4}}{2}$, $C_1^\star=\frac{(x_1-\Lambda_{\alpha,\beta})-e^{\lambda_2^\star} (x_0-\Lambda_{\alpha,\beta})}{e^{\lambda_1^\star}-e^{\lambda_2^\star}}$ and $C_2^\star=\frac{e^{\lambda_1^\star}(x_0-\Lambda_{\alpha,\beta})- (x_1-\Lambda_{\alpha,\beta})}{e^{\lambda_1^\star}-e^{\lambda_2^\star}}$. A global MPTP is also plotted in Figure 1. It is clear that $\phi_1^\star(t)$ is indeed a solution to \eqref{EL-concrete} with optimal initial and final velocities being $y_0=\dot{\phi_1^\star}(0)$ and $y_T=\dot{\phi_1^\star}(T)$.  

\begin{figure}
 \centering
\includegraphics[width=0.5\textwidth]{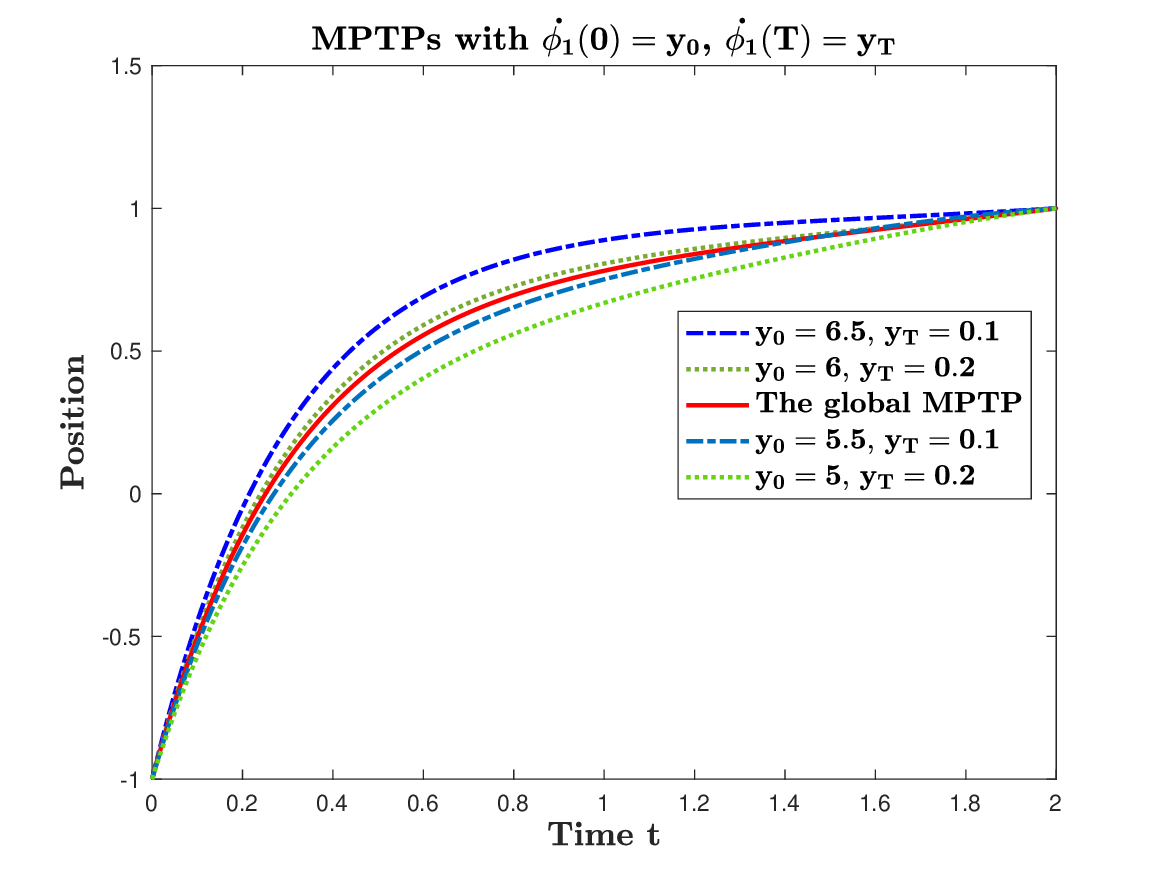}
 \label{FFig1} 
\vspace{-0.5cm}   \caption{ (Color online) MPTPs from $X(0)=-1$ to $X(2)=1$ under different initial velocity $y_0$ and final velocity $y_T$: $T=2$, $\gamma=3$, $\alpha=\beta=\frac{1}{2}$ and thus $\Lambda_{\alpha,\beta}\approx0.3989$. Red solid line: the global MPTP with $y_0=5.8078$ and $y_T=0.1904$ for which the action functional reaches $I_{\text{min}}=-\frac{\gamma T}{2}=-3$.  }
\end{figure}

\begin{figure}
 \centering
\includegraphics[width=0.95\textwidth]{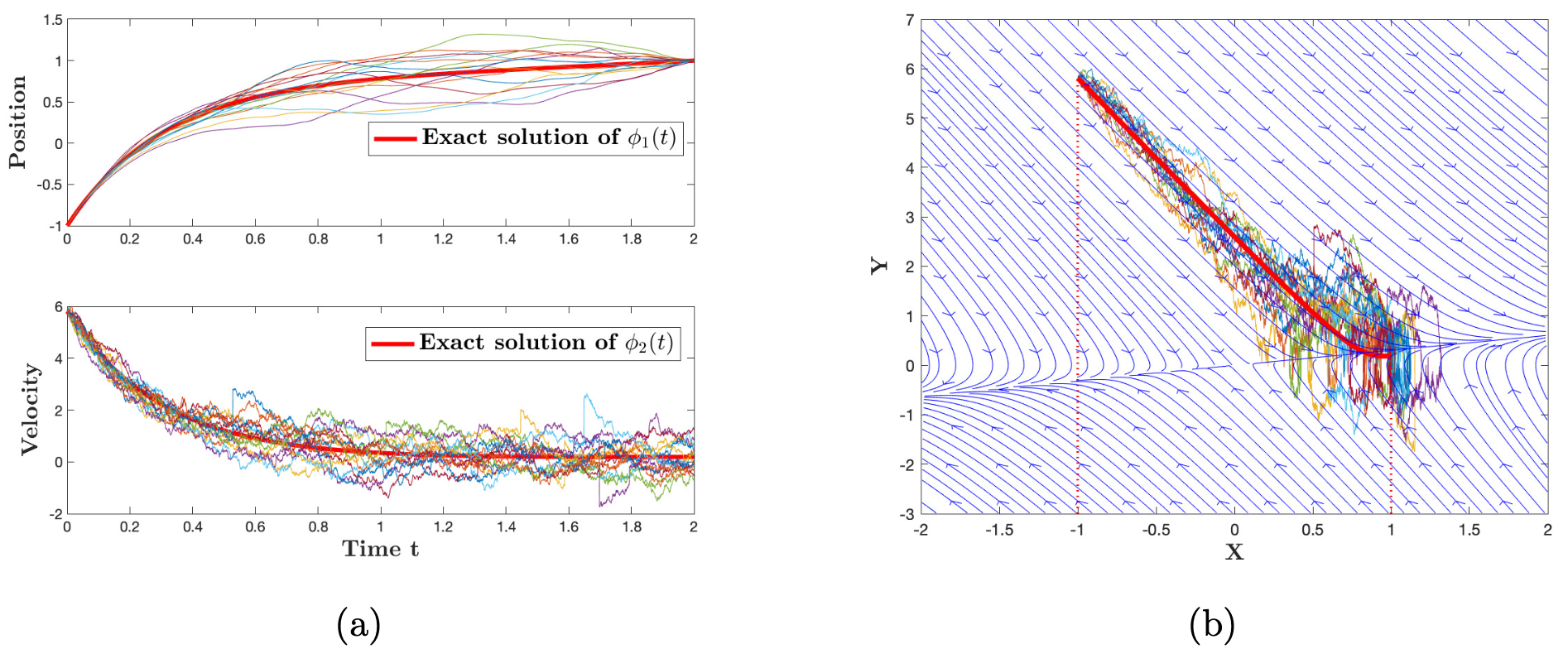}
 \label{fig2:subfig:a}
  \vspace{-0.5cm}   \caption{ (Color online)  Let $T=2$, $\gamma=3$, $\mu=0.8$, $\alpha=\beta=\frac{1}{2}$ and thus $\Lambda_{\alpha,\beta}\approx0.3989$. (a) The patterns of sample paths and the MPTP. The trajectories of $X$ are shown as the top of figure and the trajectories of $Y$ are shown as the bottom of figure.  (b) Around 15 simulations of system (\ref{Langevin1}) with initial value $X(0) = -1$ and final value  $X(2)=1$ are shown. It is evident that simulations are more likely to accumulate around the MPTP (red line).}
\end{figure}

Getting back to the settings in phase space, the MPTP between states $(x_0,y_0)$ and $(x_T,y_T)$ is just $\phi(t)=(\phi_1(t),\phi_2(t))$ with $\phi_1(t)$ given in \eqref{example-solution1} and $\phi_2(t)=\dot{\phi_1}(t)=\sum_{i=1}^4\lambda_iC_ie^{\lambda_it}$. Focusing on the globally optimal case, we investigate the time-evolution sample paths and the MPTP for the components $X$ and $Y$ respectively, as shown in Figure 2(a). We can observe that both $\phi_1(t)$ and $\phi_2(t)$ stay in the intermediate region of simulations and this is in good agreement with our theory as expected. The simulations of $X_t$ looks less ``noisy" than the ones of $Y_t$, as the noise is degenerate and, with the same time step, the random effect in the direction of $X$ is smaller than that in the direction of $Y$ when we simulate the solution to the whole SDE. In Figure 2(b), we plot the MPTP, phase paths as well as the deterministic vector field for our system in the $X$-$Y$ plane. Again, we can observe that the simulations are more likely to concentrate around the MPTP. Additionally, one interesting phenomenon is that the MPTP may choose the deterministic flow to transit, which seems somehow reasonable. The results above provide empirical evidence for supporting our theory. \par

We remark that, in most situations, the ODE system in the form of \eqref{HP-Lagevin} or \eqref{EL-concrete}, not to mention the more general one in \eqref{Hamilton-Pontryagin},  could not be sloved analytically, and nor could the constrained optimization problem. The concrete numerical method to solve this problem is not covered in this paper. We only point out that we have to be careful as discretization scheme matters and we still want to use the least action principle. Some interesting methods, e.g., stochastic variational intergrator, can be found in \cite{BouRabee2009,BouRabee20092} and the references therein.

%It is of interest to understand how a particle under random perturbations goes over a potential barrier. The analytical solver shows the optimal way is not solely by initial momentum nor solely noise
%random perturbation, but instead the joint effort, which is independent on temperature, but dependent on inertia, damping coefficient and transition time.

\section{Conclusion}\label{sec:5}
In this work, we have developed the Onsager--Machlup theory for a class of degenerate stochastic dynamical systems with both Brownian noise and L\'evy noise. Regarding the OM function as a Lagrangian, we have characterized the MPTP as the solution of the corresponding Hamilton--Pontryagin equation, under constrained conditions. The results of our work are valid for SDEs with a pure jump L\'evy process with jump measure $\nu$, as long as it has bounded variation or equivalently the integral $\int_{|\xi|<1} \xi \nu(d \xi)$ is finite (e.g., the $\alpha$-stable L\'evy motion with $0<\alpha<1$). In addition, these results can be also generalized to high-dimensional case. Numerical experiments for a second-order and underdamped kinetic Langevin system validated our theory. Particularly, for this case, the MPTP can be obtained by solving the boundary value problem of a fourth-order Euler–Lagrange equation, and the MPTP between configurations is also discussed.

Compared with our previous work \cite{Chao2019}, the common idea to derive the OM function is mainly using the Brownian motion to absorb the drift
with the help of Girsanov transformation. Due to the degeneration of noise, the so-called quasi-translation invariant measure does not exist and the proof for non-degenerate case can not be applied directly. But fortunately, the small ball probability of the $x$-component (i.e., the one whose noise term is degenerate) can be controlled by the $y$-component (i.e., the other one with a noise term). Taking advantage of this helpful technique and constructing a special auxiliary process satisfying SDE \eqref{Auxiliary-Eq}, we can overcome the difficulties raised by the degenerate noise and establish a degenerate version of the OM theory. Furthermore, different from the previous work interpreting the MPTP as a solution to the classical Euler–Lagrange equation, the MPTP here only satisfies a Hamilton–Pontryagin one containing a parametric function, which can be regarded a general Euler–Lagrange equation in the implicit
form.

\section*{Acknowledgements}
The research of YC was partially supported by NSFC grant 12101484, by the Fundamental Research Funds for the Central Universities xzy012022001, and by the NSFC grants
12271424 and 12371276. The research of PW was partially supported by CPSF grants 2022TQ0009 and 2022M720264, and by National Key R\&D Program of China 2020YFA0712800.

%\bibliographystyle{plain}
%\bibliography{yingreferences}

\bibliographystyle{alpha}
\bibliography{YC_Refs}

\begin{thebibliography}{BRKF10}

\bibitem[AB99]{Aihara1999mortensen}
S.~Aihara and A.~Bagchi.
\newblock On the mortensen equation for maximum likelihood state estimation.
\newblock {\em IEEE Trans. Automat. Control}, 44(10):1955--1961, 1999.

\bibitem[ABW10]{Albeverio2010}
S.~Albeverio, Z.~Brze\'{z}niak, and J.-L. Wu.
\newblock Existence of global solutions and invariant measures for stochastic
  differential equations driven by {P}oisson type noise with non-{L}ipschitz
  coefficients.
\newblock {\em J. Math. Anal. Appl.}, 371(1):309--322, 2010.

\bibitem[App09]{Applebaum2009}
D.~Applebaum.
\newblock {\em L\'{e}vy processes and stochastic calculus}, volume 116 of {\em
  Cambridge Studies in Advanced Mathematics}.
\newblock Cambridge University Press, Cambridge, second edition, 2009.

\bibitem[BRKF10]{Bianchi2010tempered}
M.~L. Bianchi, S.~T. Rachev, Y.~S. Kim, and F.~J. Fabozzi.
\newblock Tempered stable distributions and processes in finance: numerical
  analysis.
\newblock In {\em Mathematical and statistical methods for actuarial sciences
  and finance}, pages 33--42. Springer Italia, Milan, 2010.

\bibitem[BRM09]{BouRabee2009}
N.~Bou-Rabee and J.~E. Marsden.
\newblock Hamilton-{P}ontryagin integrators on {L}ie groups. {I}.
  {I}ntroduction and structure-preserving properties.
\newblock {\em Found. Comput. Math.}, 9(2):197--219, 2009.

\bibitem[BRO09]{BouRabee20092}
N.~Bou-Rabee and H.~Owhadi.
\newblock Stochastic variational integrators.
\newblock {\em IMA J. Numer. Anal.}, 29(2):421--443, 2009.

\bibitem[Br{\"o}19]{Brocker2019correct}
J.~Br{\"o}cker.
\newblock What is the correct cost functional for variational data
  assimilation?
\newblock {\em Climate dynamics}, 52:389--399, 2019.

\bibitem[BRT02]{Bardina2002asymptotic}
X.~Bardina, C.~Rovira, and S.~Tindel.
\newblock Asymptotic evaluation of the {P}oisson measures for tubes around jump
  curves.
\newblock {\em Appl. Math.}, 29:145--156, 2002.

\bibitem[Cap95]{Capitaine1995onsager}
M.~Capitaine.
\newblock Onsager-{M}achlup functional for some smooth norms on {W}iener space.
\newblock {\em Probab. Theor. Relat. Fields}, 102:189--201, 1995.

\bibitem[CD19]{Chao2019}
Y.~Chao and J.~Duan.
\newblock The {O}nsager-{M}achlup function as {L}agrangian for the most
  probable path of a jump-diffusion process.
\newblock {\em Nonlinearity}, 32(10):3715, 2019.

\bibitem[DB78]{Durr1978}
D.~D{\"u}rr and A.~Bach.
\newblock The onsager-{M}achlup function as {L}agrangian for the most probable
  path of a diffusion process.
\newblock {\em Commun. Math. Phys.}, 60:153--170, 1978.

\bibitem[Dit99]{Ditlevsen1999observation}
P.~D. Ditlevsen.
\newblock Observation of $\alpha$-stable noise induced millennial climate
  changes from an ice-core record.
\newblock {\em Geophys. Res. Lett.}, 26(10):1441--1444, 1999.

\bibitem[Dua15]{Duan2015introduction}
J.~Duan.
\newblock {\em An introduction to stochastic dynamics}, volume~51 of {\em
  Cambridge Texts in Applied Mathematics}.
\newblock Cambridge University Press, New York, 2015.

\bibitem[ERVE02]{Weinan2002string}
W.~E, W.~Ren, and E.~Vanden-Eijnden.
\newblock String method for the study of rare events.
\newblock {\em Phys. Rev. B}, 66(5):052301, 2002.

\bibitem[FK82]{Fujita1982onsager}
T.~Fujita and S.~Kotani.
\newblock The {O}nsager-{M}achlup function for diffusion processes.
\newblock {\em J. Math. Kyoto Univ.}, 22(1):115--130, 1982.

\bibitem[HC21]{Hu2021transition}
J.~Hu and J.~Chen.
\newblock Transition pathways for a class of high dimensional stochastic
  dynamical systems with {L}{\'e}vy noise.
\newblock {\em Chaos}, 31(6), 2021.

\bibitem[HD20]{Hu2020}
J.~Hu and J.~Duan.
\newblock Onsager-{M}achlup action functional for stochastic partial
  differential equations with {L}\'{e}vy noise.
\newblock {\em arXiv: 2011.09690}, 2020.

\bibitem[HDSX14]{Hao2014tumor}
M.~Hao, J.~Duan, R.~Song, and W.~Xu.
\newblock Asymmetric non-{G}aussian effects in a tumor growth model with
  immunization.
\newblock {\em Appl. Math. Model.}, 38(17-18):4428--4444, 2014.

\bibitem[Ish23]{Ishikawa2023stochastic}
Y.~Ishikawa.
\newblock {\em Stochastic calculus of variations---for jump processes},
  volume~54 of {\em De Gruyter Studies in Mathematics}.
\newblock De Gruyter, Berlin, third edition, 2023.

\bibitem[IW89]{Ikeda2014stochastic}
N.~Ikeda and S.~Watanabe.
\newblock {\em Stochastic differential equations and diffusion processes},
  volume~24 of {\em North-Holland Mathematical Library}.
\newblock North-Holland Publishing Co., Amsterdam; Kodansha, Ltd., Tokyo,
  second edition, 1989.

\bibitem[Kun19]{Kunita2019}
H.~Kunita.
\newblock {\em Stochastic flows and jump-diffusions}, volume~92 of {\em
  Probability Theory and Stochastic Modelling}.
\newblock Springer, Singapore, 2019.

\bibitem[LG24]{Liu2024onsager}
S.~Liu and H.~Gao.
\newblock The onsager-{M}achlup action functional for degenerate stochastic
  differential equations in a class of norms.
\newblock {\em Statist. Probab. Lett.}, 206:110009, 2024.

\bibitem[MN02]{Moret2002onsager}
S.~Moret and D.~Nualart.
\newblock Onsager-machlup functional for the fractional {B}rownian motion.
\newblock {\em Probab. Theor. Relat. Fields}, 124(2):227--260, 2002.

\bibitem[OM53]{Onsager1953}
L.~Onsager and S.~Machlup.
\newblock Fluctuations and irreversible processes.
\newblock {\em Phys. Rev.}, 91(6):1505, 1953.

\bibitem[Ria72]{Riahi1972lagrangians}
F.~Riahi.
\newblock On lagrangians with higher order derivatives.
\newblock {\em Am. J. Phys.}, 40(3):386--390, 1972.

\bibitem[SBCD15]{Schwartz2015noise}
I.~B. Schwartz, L.~Billings, T.~W. Carr, and M.~Dykman.
\newblock Noise-induced switching and extinction in systems with delay.
\newblock {\em Phys. Rev. E}, 91(1):012139, 2015.

\bibitem[Str71]{Stratonovich1971probability}
R.~L. Stratonovich.
\newblock On the probability functional of diffusion processes.
\newblock {\em Selected Trans. in Math. Stat. Prob}, 10:273--286, 1971.

\bibitem[SX20]{Song2020}
R.~Song and L.~Xie.
\newblock Well-posedness and long time behavior of singular {L}angevin
  stochastic differential equations.
\newblock {\em Stochastic Process. Appl.}, 130(4):1879--1896, 2020.

\bibitem[SX23]{Song2023}
R.~Song and L.~Xie.
\newblock Weak and strong well-posedness of critical and supercritical {SDE}s
  with singular coefficients.
\newblock {\em J. Differential Equations}, 362:266--313, 2023.

\bibitem[SZ92]{Shepp1992note}
L.~A. Shepp and O.~Zeitouni.
\newblock A note on conditional exponential moments and {O}nsager-{M}achlup
  functionals.
\newblock {\em Ann. Probab.}, 20:652--654, 1992.

\bibitem[YM06a]{Yoshimura20061}
H.~Yoshimura and J.~E. Marsden.
\newblock Dirac structures in {L}agrangian mechanics. {I}. {I}mplicit
  {L}agrangian systems.
\newblock {\em J. Geom. Phys.}, 57(1):133--156, 2006.

\bibitem[YM06b]{Yoshimura20062}
H.~Yoshimura and J.~E. Marsden.
\newblock Dirac structures in {L}agrangian mechanics. {II}. {V}ariational
  structures.
\newblock {\em J. Geom. Phys.}, 57(1):209--250, 2006.

\bibitem[ZSDK16]{Zheng2016transitions}
Y.~Zheng, L.~Serdukova, J.~Duan, and J.~Kurths.
\newblock Transitions in a genetic transcriptional regulatory system under
  {L}{\'e}vy motion.
\newblock {\em Sci. Rep.}, 6(1):29274, 2016.

\end{thebibliography}
\end{document}